\documentclass[12pt]{amsart}
\usepackage{amsmath, amssymb, amsthm, bm, mathtools, enumitem, caption}
\usepackage{hyperref}
\usepackage{graphicx, tikz}
\usepackage{float}
\usepackage[margin=1in]{geometry}
\usepackage{xcolor}
\theoremstyle{plain}
\newtheorem{theorem}{Theorem}[section]
\newtheorem{proposition}[theorem]{Proposition}
\newtheorem{lemma}[theorem]{Lemma}

\newtheorem*{conjecture}{Conjecture}
\newtheorem*{question}{Question}
\newtheorem*{examples}{Examples}
\newtheorem*{prediction}{Stochastic Prediction}
\newtheorem*{problem}{Main Problem}

\theoremstyle{definition}
\newtheorem*{definition}{Definition}

\theoremstyle{remark}
\newtheorem*{remark}{Remark}

\newcommand{\Ddigits}{\mathcal{D}}
\newcommand{\1}{\mathbf{1}}
\newcommand{\abs}[1]{\left|#1\right|}

\usepackage{fancyhdr}

\title{Dead ends in square-free digit walks}
\newif\ifmanyauthors
\manyauthorstrue 

\ifmanyauthors
  \usepackage[foot]{amsaddr}
  \usepackage{textcomp}
  \newcommand{\dmd}{\ensuremath{\diamond}}

  \author{Evan Chen\textsuperscript{\dag}}
  \email{evan@axiommath.ai}

  \author{Chris Cummins\textsuperscript{*}}
  \email{chris@axiommath.ai}

  \author{Ben Eltschig\textsuperscript{*}}
  \email{ben@axiommath.ai}

  \author{Dejan Grubisic\textsuperscript{*}}
  \email{dejan@axiommath.ai}

  \author{Leopold Haller\textsuperscript{*}}
  \email{leo@axiommath.ai}

  \author{Letong Hong\textsuperscript{\dmd}}
  \email{carina@axiommath.ai}

  \author{Andranik Kurghinyan\textsuperscript{*}}
  \email{andranik@axiommath.ai}

  \author{Kenny Lau\textsuperscript{\dag*}}
  \email{kenny@axiommath.ai}

  \author{Hugh Leather\textsuperscript{*}}
  \email{hugh@axiommath.ai}

  \author{Seewoo Lee\textsuperscript{\dag}}
  \email{seewoo@axiommath.ai}

  \author{Aram Markosyan\textsuperscript{*}}
  \email{am@axiommath.ai}

  \author{Ken Ono\textsuperscript{\dag}}
  \email{ken@axiommath.ai}

  \author{Manooshree Patel\textsuperscript{*}}
  \email{manooshree@axiommath.ai}

  \author{Gaurang Pendharkar\textsuperscript{*}}
  \email{gaurang@axiommath.ai}

  \author{Vedant Rathi\textsuperscript{*}}
  \email{vedant@axiommath.ai}

  \author{Alex Schneidman\textsuperscript{*}}
  \email{alex@axiommath.ai}

  \author{Volker Seeker\textsuperscript{*}}
  \email{volker@axiommath.ai}

  \author{Shubho Sengupta\textsuperscript{\dmd}}
  \email{shubho@axiommath.ai}

  \author{Ishan Sinha\textsuperscript{*}}
  \email{ishan@axiommath.ai}

  \author{Jimmy Xin\textsuperscript{*}}
  \email{jimmy@axiommath.ai}

  \author{Jujian Zhang\textsuperscript{\dag*}}
  \email{jujian@axiommath.ai}

\else

  \author{Evan Chen}
  \email{evan@axiommath.ai}

  \author{Kenny Lau}
  \email{kenny@axiommath.ai}

  \author{Seewoo Lee}
  \email{seewoo@axiommath.ai}

  \author{Ken Ono}
  \email{ken@axiommath.ai}

  \author{Jujian Zhang}
  \email{jujian@axiommath.ai}
\fi

\address{Axiom Math, 124 University Avenue, Palo Alto, CA 94301}

\usepackage[obeyFinal,textsize=tiny,shadow,loadshadowlibrary]{todonotes}
\usepackage{soul}
\newcommand{\added}[1]{\ifmmode{\color{red}#1}\else{\color{red}\ul{#1}}\fi}

\begin{document}
\maketitle

\ifmanyauthors
\begin{center}
  \footnotesize
  Authors are listed alphabetically. \\
  \textsuperscript{\dag}Mathematical contributor,
  \textsuperscript{*}Engineering contributor,
  \textsuperscript{\dmd}Principal investigator.
\end{center}
\fi

\begin{abstract}
We study ``dead ends'' in square-free digit walks: square-free integers $N$ such that, in base $b$,
every one-digit extension $bN+d$ is non-square-free.
In base $10$, the stochastic independence model of ~\cite{MillerEtAl} suggests
that infinite square-free walks occur with probability near $1,$ corresponding
to an asymptotic dead-end density of $\approx 5.218\times 10^{-5}$.
We prove that the true asymptotic dead-end density satisfies
\[ c_{\mathrm{dead}} \approx 1.317\times 10^{-9}, \]
roughly a factor of $\sim\! 4\times 10^4$ smaller than the prediction.
For every base $b\geq 2$, we prove that dead-end densities exist and are given
by a closed-form expression (as a finite alternating sum of Euler products).
The argument is fully formalized in Lean/Mathlib, and was produced automatically by AxiomProver
from a natural-language statement of the problem.
\end{abstract}

\section*{Update}
After we posted the first version of this manuscript on the arXiv,
we learned from Kannan Soundararajan that the result in this paper was
previously obtained by Mirsky in 1947 \cite{MIRSKY}.
We mistakenly assumed that the 2024 paper of Miller et al.~\cite{MillerEtAl}
accurately represented the status of this question.
Their paper makes no reference to Mirsky \cite{MIRSKY};
that work seems to have been forgotten.
Although our Lean formalization still stands,
this manuscript will not be submitted for journal publication.

\section{Introduction}
For a positive integer $N$ and a digit $d\in\{0,1,\dots,9\}$, define the
\emph{right-append map}
\[
T_d(N) := 10N+d.
\]
A positive integer is \emph{square-free} if it is not divisible by any perfect square $>1$.
We study \textit{square-free walks}.
Starting with a square-free integer $N_0$,
we attempt to form a sequence $N_0, N_1, N_2, \dots$ by successively appending
digits such that every term in the sequence remains square-free.
For example, starting with $N=5$, we have
\begin{displaymath}
\begin{split}
      & T_6(5) = 56 = 2^3\cdot 7 \quad \xrightarrow{\text{Fail}} \qquad \ \ \ \
      (\text{Divisible by the square $2^2$}),\\
      & T_1(5) = 51 =3\cdot 17  \quad \xrightarrow{\text{Success}} \qquad (\text{Square-free}).
\end{split}
\end{displaymath}
So although $T_6(5)$ fails, we can continue with $51$ and we find that
\begin{displaymath}
\begin{split}
  &T_3(51) = 513 = 3^3\cdot 19 \quad \xrightarrow{\text{Fail}} \qquad \ \ \ \
  (\text{Divisible by the square $3^2$}),\\
  &T_9(51) = 519 =3\cdot 173  \quad \xrightarrow{\text{Success}} \qquad (\text{Square-free}).
\end{split}
\end{displaymath}
Starting with any $N_0$, this process defines a tree,
where branches terminate when appending a digit results in a non-square-free integer.
This process inspires a natural question.

\begin{question}
Is it possible to continue such a walk indefinitely?
In other words, can we walk to infinity along the square-free integers?
Are there trees with infinitely long branches?
\end{question}

This question is raised in \cite{MillerEtAl} by Miller et al., and the following conjecture is posed.

\begin{conjecture}[Miller et al.\ \cite{MillerEtAl}]
There exists an infinite sequence of square-free integers
\[ \{N_0,N_1,N_2, N_3, \dots\} \]
and digits
$d_{k+1}\in\{0,1,\dots,9\}$ such that $N_{k+1}=10N_k+d_{k+1}$ for all $k\ge 0$.
\end{conjecture}

\begin{remark}
The paper by Miller et al. \cite{MillerEtAl} considers further walks. A recent paper by Kominers \cite{Kominers} proves that similar walks are impossible in the setting of Fibonacci numbers and Lucas numbers.
\end{remark}

Despite its innocent appearance, this conjecture remains open.
In particular, no one has exhibited an explicit infinite sequence of digits that
provably generates square-free integers at every step, nor has the existence of such a walk been established.

Miller et al.\ \cite{MillerEtAl} analyze this conjecture using a stochastic
framework called the ``Blind Unlimited Model,'' where square-freeness is treated
as a random event with independent probability $6/\pi^2$,
the asymptotic density of square-free integers among the positive integers.
They model the 10 potential digit extensions as a branching process,
and they calculate that the probability of the entire tree going extinct is only
$\ell \approx 8.6 \times 10^{-5}$.
This implies that an infinite walk from a given $N_0$ exists with probability $\approx 0.99991$.
To make this precise, we make the following definition.

\begin{definition}
A positive integer $N$ is a \emph{dead end} if
\begin{enumerate}
    \item $N$ is square-free, and
    \item for every digit $d \in \{0, 1, \dots, 9\}$, the integer $10N+d$ is \emph{not} square-free.
\end{enumerate}
\end{definition}

\noindent
The stochastic heuristic predicts the asymptotic density of dead ends in terms of
$\rho:=\frac{6}{\pi^2}$ and
the counting function
\begin{equation}
D(X) := \# \left\{ 1\leq N \leq X \ : \ N \text{ is a dead end} \right\}.
\end{equation}

\begin{prediction}[Miller et al.\ \cite{MillerEtAl}]
If we define $P_k$ to be the probability that the rooted digit-walk tree generated
from a \emph{square-free} root has \emph{finite height at most $k$},
then we have:
\begin{enumerate}
\item $P_1$ is the probability that the root has no square-free children (no digit can be added),
so
\[
P_1=(1-\rho)^{10}.
\]
\item More generally, the tree has height $\le k+1$ exactly when, for each digit $d$,
either $10N+d$ is not square-free (probability $1-\rho$), or else $10N+d$ is square-free (probability $\rho$)
\emph{and} the subtree rooted at $10N+d$ has height $\le k$ (probability $P_k$).
By the independence heuristic across digits, this yields the recursion
\[
P_{k+1}=\bigl( (1-\rho)+\rho P_k \bigr)^{10}\qquad (k\ge 1).
\]
\end{enumerate}
The ``extinction probability'' (i.e.\ the probability that the entire tree is finite, equivalently that
there is no infinite square-free walk) is predicted to be the limit
\[
\ell:=\lim_{k\to\infty} P_k\approx 8.59502\times 10^{-5}.
\]
Taking into account the possibility that $N_0$ might not be square-free,
the stochastic prediction for the asymptotic density for dead ends is
\[
\lim_{X\to +\infty} \frac{D(X)}{X} \approx \frac{6}{\pi^2} \cdot P_1 \approx 5.21818\times 10^{-5}.
\]
\end{prediction}

Although dead ends are expected to be very rare, occurring with near zero density,
it is possible to show explicit examples.
Indeed,  Miller et al.\ offer the example $N=231546210170694222$, where we find that
\begin{align*}
    231546210170694222{\bf 0} &= 2^2\cdot 578865525426735555,\\
    231546210170694222{\bf 1} &= 11^2\cdot 19136050427330101,\\
    231546210170694222{\bf 2} &= 19^2\cdot 6414022442401502,\\
    231546210170694222{\bf 3} &= 7^2\cdot 47254328606264127,\\
    231546210170694222{\bf 4} &= 2^4\cdot 144716381356683889,\\
    231546210170694222{\bf 5} &= 5^2\cdot 92618484068277689,\\
    231546210170694222{\bf 6} &= 3^3\cdot 85757855618775638,\\
    231546210170694222{\bf 7} &= 13^2\cdot 13700959181697883,\\
    231546210170694222{\bf 8} &= 2^2\cdot 578865525426735557,\\
    231546210170694222{\bf 9} &= 17^2\cdot 8011979590681461.
\end{align*}
This motivates the following concrete, \emph{non-stochastic} question about actual integers.
Exactly how rare are dead ends?
Is the stochastic prediction correct?
Despite the convincing evidence that such stochastic models are accurate for many questions,
here we prove that stochastic reasoning gives the wrong answer when applied to dead ends.
Dead ends are orders of magnitude rarer than the stochastic prediction.

For convenience, we define $\Ddigits:=\{0, 1, 2,\dots, 9\}$ to be the set of base $10$ digits.

\begin{theorem}[Asymptotic density of dead ends]\label{thm:main}
For all $X\ge 3$, we have
\[
D(X)=c_{\mathrm{dead}}\,X + O\!\left(\frac{X}{\sqrt{\log X}}\right),
\]
where the constant is
\[
c_{\mathrm{dead}}
:=\sum_{S\subseteq \Ddigits} (-1)^{\abs{S}}
\prod_{p \ {\text prime}}\left(1-\frac{\nu_p(S)}{p^2}\right),
\]
and where for each subset $S\subseteq\Ddigits$ we define
\[
\nu_p(S):=\#\Bigl\{n\bmod p^2:\ p^2\mid n\ \text{or}\ p^2\mid(10n+d)\ \text{for some }d\in S\Bigr\}.
\]
Moreover, have that
$c_{\mathrm{dead}}  \approx 1.3170\times 10^{-9}.$
\end{theorem}

\begin{remark}
The main result of this paper shows that the stochastic
prediction is far from the truth for dead ends.
The reason is arithmetic. For each prime $p\ge 7$,
the condition $p^2\mid (10n+d)$ forces $n$ into a \emph{single}
residue class modulo $p^2$, so one prime square can ``explain'' at most one digit obstruction at a time.
For all ten digits to fail square-freeness simultaneously, many distinct prime squares must be involved,
creating strong dependence that the naive stochastic model misses.
\end{remark}

While Theorem~\ref{thm:main} focuses on base $10$, the phenomenon of unavoidable dead ends is universal.
The proof of Theorem~\ref{thm:main} generalizes easily {\it mutatis mutandis} to
give a general theorem for all bases $b\geq 2$.
Specifically, we can restate the problem for general $b \ge 2$ as follows.
\begin{problem}
  Fix an integer base $b\geq 2$, and let
  \[
  \mathcal{D}_b := \{0,1,2,\dots,b-1\}.
  \]
  A positive integer $N$ is a \emph{base $b$ dead end} if $N$ is square-free and, for every digit
  $d\in \mathcal{D}_b$, the integer $bN+d$ is not square-free. Define the counting function
  \[
  D_b(X) \;:=\; \#\{1\le N\le X : N \text{ is a base-$b$ dead end}\}.
  \]
  Determine, with proof, the constant $c_{\mathrm{dead}}(b)$ such that
  \[ D_b(X) \;=\; c_{\mathrm{dead}}(b)\,X  + o_b(X). \]
\end{problem}

Namely, we have following general theorem.

\begin{theorem}[Asymptotic density of dead ends in base $b$]\label{thm:base_b}
For all $X\ge 3$, we have
\[
D_b(X) \;=\; c_{\mathrm{dead}}(b)\,X \;+\; O_b\!\left(\frac{X}{\sqrt{\log X}}\right),
\]
where $c_{\mathrm{dead}}(b)$ is the well-defined positive constant
\[
c_{\mathrm{dead}}(b)
\;:=\;
\sum_{S\subseteq \mathcal{D}_b} (-1)^{|S|}
\prod_{p\ \mathrm{prime}}
\left(1-\frac{\nu_{p,b}(S)}{p^2}\right),
\]
and for each prime $p$ and each subset $S\subseteq \mathcal{D}_b$ we define
\[
\nu_{p,b}(S)
\;:=\;
\#\Bigl\{ n \bmod p^2 :\ p^2\mid n\ \text{or}\ p^2\mid (bn+d)\ \text{for some } d\in S \Bigr\}.
\]
\end{theorem}

\begin{examples} Using the formulas in Theorem~\ref{thm:base_b}, we computed approximations for the dead-end densities for the prime bases $b\in \{2, 3, 5, 7\}.$

\begin{table}[H]
\centering
\renewcommand{\arraystretch}{1.2}
\setlength{\tabcolsep}{6pt}
\begin{tabular}{|c|c|c|c|c|c|c|c|c|c|}
\hline
$b$ & $2$ & $3$  & $5$  & $7$  \\
\hline
$c_{\mathrm{dead}}(b)$
& $4.13253\times 10^{-2}$
& $9.44842\times 10^{-3}$
& $8.16352\times 10^{-5}$
& $3.08003\times 10^{-6}$\\
\hline
\end{tabular}
\caption{Dead-end density constant $c_{\mathrm{dead}}(b)$ for primes $b\in \{2, 3, 5, 7\}$.}
\label{tab:cdead-b}
\end{table}
\end{examples}

\begin{remark}
This work is a case study and test case for AxiomProver, an AI tool currently under development,
aimed at end-to-end automated theorem proving in mainstream mathematics.
We provided a natural-language formulation of the Main Problem.
AxiomProver correctly discovered the formula for $c_{\mathrm{dead}}(b)$ without any further input,
then generated a Lean/Mathlib statement and a fully verified proof.
Using that formal development as a reference point,
we prepared the exposition in the main text for a mathematical audience, aiming to supply context,
motivation, and a streamlined derivation that can be read independently of the Lean code.
\end{remark}

This paper is organized as follows.
In Section~\ref{sec:proofs} we recall a few facts from elementary number theory
that we then employ to prove Theorem~\ref{thm:main}.
We also give a sketch of the proof of Theorem~\ref{thm:base_b}.
In Section~\ref{sec:AxiomProver} we document the formalization,
clarifies the experimental conditions under which the system was evaluated,
and provides links to the relevant files for interested readers
(Mathematicians not interested in automated theorem proving
can thus safely ignore Section~\ref{sec:AxiomProver}.) Finally,
in the Appendix we give formulas that one can use to compute $c_{\mathrm{dead}}$ to arbitrary precision.

\section*{Acknowledgements}
The authors thank Steven J.\ Miller and Scott Kominers for their comments on an earlier version of this paper. We thank Kannan Soundararajan
for informing us that the result in this paper was previously obtained
by Mirsky in 1947 \cite{MIRSKY} after the first version of this manuscript
was posted to the arXiv.
\ifmanyauthors
\else
This paper describes a test case for AxiomProver,
an autonomous system that is currently under development.
The project engineering team is
Chris Cummins,
Ben Eltschig,
GSM,
Dejan Grubisic,
Leopold Haller,
Letong Hong (principal investigator),
Andranik Kurghinyan,
Kenny Lau,
Hugh Leather,
Aram Markosyan,
Manooshree Patel,
Gaurang Pendharkar,
Vedant Rathi,
Alex Schneidman,
Volker Seeker,
Shubho Sengupta (principal investigator),
Ishan Sinha,
Jimmy Xin,
and Jujian Zhang.
\fi

\section{Proof of Theorems~\ref{thm:main} and \ref{thm:base_b}}\label{sec:proofs}
This section collects the precise tools we will use.
Everything here is standard elementary multiplicative number theory (for example,  see \cite{Apostol}).
We require  the M\"obius function $\mu:\mathbb{Z}_{\ge 1}\to\{-1,0,1\}$
which is defined by $\mu(1)=1$ and for $n\geq 2$ by
\[
\mu(n)=
\begin{cases}
0, & \text{if $p^2\mid n$ for some prime $p$,}\\[4pt]
(-1)^k, & \text{if $n=p_1p_2\cdots p_k$ is a product of $k$ distinct primes.}
\end{cases}
\]
This function has many properties, and here we make use of the simple fact that
for every integer $n\ge 1$, we have
\begin{equation}\label{sqrfree}
\mu(n)^2
=\begin{cases} 1 \ \ \ \ &{\text { if $n$ is square-free}},\\ 0 \ \ \ \ &{\text { otherwise}} \end{cases}.
\end{equation}
In other words, the square of the M\"obius function is the square-free indicator function.

We will repeatedly use the elementary counting fact that for fixed integers $M\ge 1$ and $a$,
the number of integers $1\le n\le X$ with $n\equiv a\pmod M$ equals $\frac{X}{M}+O(1)$,
with an absolute implied constant.
We shall also make critical use of this form of the Chinese Remainder Theorem.

\begin{lemma}[Chinese remainder theorem (CRT)]\label{lem:crt}
Let $M_1,\dots,M_k$ be pairwise coprime moduli.
For each $i$, let $A_i$ be a set of residue classes modulo $M_i$.
Then the number of residue classes modulo $M=M_1\cdots M_k$ whose reduction
modulo $M_i$ lies in $A_i$ for all $i$
is exactly $\prod_{i=1}^k |A_i|$.
\end{lemma}

\begin{proof}
The CRT states that the reduction map
\[
\mathbb{Z}/M\mathbb{Z}\ \longrightarrow\ \prod_{i=1}^k \mathbb{Z}/M_i\mathbb{Z}
\]
is a bijection. Therefore specifying residues modulo each $M_i$ independently
specifies a unique residue modulo $M$.
Counting choices gives the product formula.
\end{proof}

\subsection{Inclusion-Exclusion Framework}

The proof of Theorems~\ref{thm:main} and \ref{thm:base_b} is based on an inclusion-exclusion argument,
combined with some analytic number theory.
For Theorem~\ref{thm:main},
we have the digit set $d\in\Ddigits:=\{0, 1, \dots, 9\}.$  Furthermore, define the statement
\[
A_d(n):=\text{``$10n+d$ is square-free''}.
\]
Therefore, $n$ is a dead end exactly when $n$ is square-free and \emph{none} of the $A_d(n)$ hold.
Using the identity $\1_{\text{$m$ square-free}}=\mu(m)^2$ (see \eqref{sqrfree}),
we obtain the dead-end indicator identity
\begin{equation}
\1_{\text{$n$ is a dead end}}
=\mu(n)^2\prod_{d\in\Ddigits}\bigl(1-\mu(10n+d)^2\bigr).
\end{equation}
Indeed, the $\mu(n)^2$ factor guarantees that $n$ is square-free,
while the vanishing of products requires each of the $10n+d$ values is non-square-free.

Expanding this product by inclusion--exclusion yields the following closed
formula for the counting function $D(X).$

\begin{proposition}[Exact digit inclusion--exclusion]\label{prop:IE}
For every $X\ge 1$, we have
\[
D(X)=\sum_{S\subseteq \Ddigits} (-1)^{\abs{S}}\,Q_S(X),
\]
where
\[
Q_S(X):=\sum_{n\le X}\mu(n)^2\prod_{d\in S}\mu(10n+d)^2.
\]
\end{proposition}

\begin{proof}
Expand $\prod_{d\in\Ddigits}(1-\mu(10n+d)^2)$ into a sum over subsets $S\subseteq\Ddigits$,
where choosing $d\in S$ contributes the factor $-\mu(10n+d)^2$ and choosing $d\notin S$ contributes $1$.
Then sum over $n\le X$.
\end{proof}

Therefore, the proof of Theorem~\ref{thm:main} reduces to asymptotics for $Q_S(X)$ for each fixed $S$.

\subsection{Asymptotics for $Q_S(X)$}

Fix a subset $S\subseteq\Ddigits$ and write $r=\abs{S}$.
Then $Q_S(X)$ counts those integers $n\le X$ for which $n$ and all one-digit extensions $10n+d$ ($d\in S$)
are square-free:
\begin{equation}\label{QSX}
  \begin{aligned}
    Q_S(X) &=\sum_{n\le X}\mu(n)^2\prod_{d\in S}\mu(10n+d)^2 \\
    &=\#\{n\le X:\ n\ \text{and}\ 10n+d \text{ are square-free for all } d \in S\}.
  \end{aligned}
\end{equation}
A positive integer is square-free if and only if it is not divisible by $p^2$ for any prime $p$.
Therefore, the condition counted by $Q_S(X)$ can be expressed prime-by-prime.
Namely, for each prime $p,$ define the set of \emph{bad} residues modulo $p^2$
\[
B_p(S):=\Bigl\{n\bmod p^2:\ p^2\mid n\ \text{or}\ p^2\mid(10n+d)\ \text{for some }d\in S\Bigr\},
\]
and let $\nu_p(S):=\abs{B_p(S)}$.

To obtain asymptotics for $Q_S(X)$, we carefully sift by finitely many prime squares.
To this end, let $z\ge 2$ and set
\[
  P(z):=\prod_{\substack{p\leq z\\ p \text{ prime}}} p \qquad {\text {\rm and}}\qquad M(z):=P(z)^2.
\]
Define the \emph{$z$--sifted} counting function
\[
Q_S(X;z):=\#\{n\le X:\ n\bmod p^2\notin B_p(S)\ \text{for every prime }p\le z\},
\]
and the corresponding finite Euler product
\begin{equation}\label{CzS}
  C_z(S):=\prod_{\substack{p\le z\\ p \text{ prime}} }\left(1-\frac{\nu_p(S)}{p^2}\right).
\end{equation}

\begin{lemma}[Counting the $z$--sifted set]\label{lem:z-sifted}
For every $X\ge 1$ and every $z\ge 2$, we have
\[
Q_S(X;z)=C_z(S)\,X+O(M(z)).
\]
\end{lemma}

\begin{proof}
For a fixed prime $p\le z$, the condition ``$n\bmod p^2\notin B_p(S)$'' excludes
exactly $\nu_p(S)$ residue classes modulo $p^2$,
so it allows exactly $p^2-\nu_p(S)$ residue classes modulo $p^2$.
Since the moduli $p^2$ are pairwise coprime as $p$ varies over primes,
the CRT (Lemma~\ref{lem:crt}) implies that the number of residue
classes modulo
\[
  M(z)=\prod_{\substack{p\le z\\ p\text{ prime}}} p^2
\]
that are simultaneously allowed for every prime $p\le z$ is exactly
\[
G_z(S):=\prod_{p\le z}\bigl(p^2-\nu_p(S)\bigr).
\]
Each such residue class modulo $M(z)$ contributes $\frac{X}{M(z)}+O(1)$ integers $n\le X$.
Summing over the $G_z(S)$ allowed classes gives
\[
Q_S(X;z)=G_z(S)\left(\frac{X}{M(z)}+O(1)\right)
=\frac{G_z(S)}{M(z)}\,X+O\!\bigl(G_z(S)\bigr).
\]
Finally, $G_z(S)\le M(z)$, so the error term is $O(M(z))$, and
\[
\frac{G_z(S)}{M(z)}=\prod_{p\le z}\left(1-\frac{\nu_p(S)}{p^2}\right)=C_z(S). \qedhere
\]
\end{proof}

The previous lemma only takes into account the primes $p\leq z.$ The next lemma
bounds the contributions to $Q_S(X)$ from the larger primes.

\begin{lemma}[Bounding the contribution of large prime squares]\label{lem:large-primes}
If $z\ge 5$, then for every $X\ge 1$ we have
\[
Q_S(X)=Q_S(X;z)+O\!\left((\abs{S}+1)\left(\frac{X}{z}+\sqrt{X}\right)\right).
\]
\end{lemma}

\begin{proof}
Any $n$ counted by $Q_S(X;z)$ but \emph{not} counted by $Q_S(X)$ has the property that
for some prime $p>z$ at least one of the integers
\[
n,\quad 10n+d\ (d\in S)
\]
is divisible by $p^2$. Since $10n+d\le 10X+9$ for $n\le X$,
such a prime must satisfy $p^2\le 10X+9$ (i.e.\ $p\le \sqrt{10X+9}$).
Since $z\ge 5$, every prime $p>z$ satisfies $p\ge 7$, and hence $10$ is invertible modulo $p^2$.
Therefore, for each $d\in S$ the congruence $p^2\mid (10n+d)$ forces $n$ into
exactly one residue class modulo $p^2$,
and the congruence $p^2\mid n$ forces $n\equiv 0\pmod{p^2}$.
In either case, the number of solutions $n\le X$ is $\frac{X}{p^2}+O(1)$.

By the union bound, the number of $n\le X$ excluded by primes $p>z$ is therefore at most
\[
  \sum_{\substack{p>z\\ p \text{ prime}}}\ \sum_{\substack{\text{forms }L\\ L\in\{n\}\cup\{10n+d:\ d\in S\}}}
\left(\frac{X}{p^2}+O(1)\right),
\]
where the prime sum may be restricted to $p\le \sqrt{10X+9}$.
There are exactly $\abs{S}+1$ forms $L$, so this is
\[
(\abs{S}+1)\left(X\sum_{z<p\le \sqrt{10X+9}}\frac{1}{p^2}\ +\ O\!\bigl(\pi(\sqrt{10X+9})\bigr)\right),
\]
where $\pi(y)$ denotes the number of primes $\le y$.
We bound the two quantities crudely but explicitly:
\[
  \sum_{\substack{z<p\le \sqrt{10X+9}\\ p \text{ prime}}}\frac{1}{p^2}\ \le\ \sum_{n>z}\frac{1}{n^2}\ \ll\ \frac{1}{z}
\qquad {\text{and}}\qquad
\pi(\sqrt{10X+9})\le \sqrt{10X+9}\ \ll\ \sqrt{X}.
\]
Combining these bounds gives the claimed error term.
\end{proof}

It is important to consider the limiting behavior of the constants in \eqref{CzS}  as $z\rightarrow +\infty.$

\begin{lemma}[Convergence of the Euler product]\label{lem:euler-conv}
The infinite product
\[
  C(S):=\prod_{p\text{ prime}}\left(1-\frac{\nu_p(S)}{p^2}\right)
\]
converges absolutely. Moreover, for every $z\ge 7$, we have
\[
C(S)=C_z(S)+O\!\left(\frac{1}{z}\right).
\]
\end{lemma}

\begin{proof}
For the primes $p\le 5$, the factors $\left(1-\frac{\nu_p(S)}{p^2}\right)$ are
well-defined real numbers in $[0,1]$.
We therefore focus on primes $p\ge 7$. Since $p\nmid 10$, the number $10$ is invertible modulo $p^2$.
Thus $p^2\mid n$ forces $n\equiv 0\pmod{p^2}$ (one residue class), and for each digit $d\in S$ the congruence
$p^2\mid(10n+d)$ forces $n$ into a single residue class modulo $p^2$.
Taking the union over $d\in S$ shows that
\[
\nu_p(S)\le 1+\abs{S}\le 11.
\]
In particular, for every $p\ge 7$ we have $0\le \nu_p(S)/p^2\le 11/49<1/2$.

Now consider the partial products over primes in $[7,N]$:
\[
P_N:=\prod_{\substack{7\le p\le N\\ p \text{ prime}}}\left(1-\frac{\nu_p(S)}{p^2}\right).
\]
Taking logarithms and using the inequality $\abs{\log(1-u)}\le 2u$ valid for $0\le u\le 1/2$, we obtain
\[
  \sum_{\substack{p\ge 7\\ p\text{ prime}}}\abs{\log\!\left(1-\frac{\nu_p(S)}{p^2}\right)}
  \le 2\sum_{\substack{p\ge 7\\ p \text{ prime}}}\frac{\nu_p(S)}{p^2}
  \ll \sum_{\substack{p\ge 7\\ p\text{ prime}}}\frac{1}{p^2}<\infty.
\]
Therefore the series $\sum_{p\ge 7}\log\!\left(1-\frac{\nu_p(S)}{p^2}\right)$ converges absolutely,
so the sequence $\log P_N$
converges as $N\to\infty$, and hence $P_N$ converges to a finite nonzero limit.
This proves absolute convergence of $C(S)$.

For the tail estimate, write
\[
  C(S)=C_z(S)\cdot \prod_{\substack{p>z\\ p \text{ prime}}}\left(1-\frac{\nu_p(S)}{p^2}\right).
\]
For $z\ge 7,$ we have $0\le \nu_p(S)/p^2\le 11/49<1/2$ for every $p>z$,
so again using $\abs{\log(1-u)}\le 2u$ gives
\[
  \abs{\log\!\left(\prod_{\substack{p>z\\ p\text{ prime}}}\left(1-\frac{\nu_p(S)}{p^2}\right)\right)}
  \le 2\sum_{\substack{p>z\\ p\text{ prime}}}\frac{\nu_p(S)}{p^2}
\ll \sum_{n>z}\frac{1}{n^2}\ll \frac{1}{z}.
\]
Exponentiating yields $\prod_{p>z}\left(1-\frac{\nu_p(S)}{p^2}\right)=1+O(1/z)$,
and hence $C(S)=C_z(S)+O(1/z)$.
\end{proof}

Assembling these observations, we obtain the final asymptotic for $Q_S(X).$

\begin{proposition}[Asymptotic for $Q_S(X)$]\label{prop:QS-asymp}
For all $X\ge 3$, we have
\[
Q_S(X)=C(S)\,X+O\!\left(\frac{X}{\sqrt{\log X}}\right),
\]
where
\[
  C(S)=\prod_{p \text{ prime}}\left(1-\frac{\nu_p(S)}{p^2}\right).
\]
\end{proposition}

\begin{proof}
Let $z:=\max\{7,\lfloor \sqrt{\log X}\rfloor\}$, so $z\ge 7$ and $z\to\infty$ as $X\to\infty$.
Combining Lemma~\ref{lem:z-sifted} with Lemma~\ref{lem:large-primes}, we obtain
\[
Q_S(X)=C_z(S)\,X+O(M(z))\ +\ O\!\left((\abs{S}+1)\left(\frac{X}{z}+\sqrt{X}\right)\right).
\]
By Lemma~\ref{lem:euler-conv}, we have $C(S)=C_z(S)+O(1/z)$, which in turn gives
\[
C_z(S)\,X = C(S)\,X + O\!\left(\frac{X}{z}\right).
\]
Combining the last two expressions gives
\[
Q_S(X)=C(S)\,X + O\!\left(M(z)+\frac{X}{z}+\sqrt{X}\right).
\]

It remains to bound $M(z)$ in terms of $X$.
Since
\[ P(z)=\prod_{p\le z} p \le z^{\pi(z)}\le z^{\,z} \]
we obtain
\[ M(z)=P(z)^2\le z^{2z}.  \]
With $z=\max\{7,\lfloor \sqrt{\log X}\rfloor\}$, we have $\log z\ll \log\log X$, so
\[
\log M(z)\le 2z\log z \ll \sqrt{\log X}\,\log\log X=o(\log X).
\]
In particular, $M(z)=o\!\left(\frac{X}{\sqrt{\log X}}\right)$ as $X\to\infty$.
Also, since $z\asymp \sqrt{\log X}$ and $\sqrt{X}=o\!\left(\frac{X}{\sqrt{\log X}}\right)$, we conclude that
\[
Q_S(X)=C(S)\,X+O\!\left(\frac{X}{\sqrt{\log X}}\right),
\]
as claimed.
\end{proof}

\subsection{Proof of Theorem \ref{thm:main}}
Finally, combine Proposition~\ref{prop:IE} with Proposition~\ref{prop:QS-asymp}.
We obtain
\[
D(X)=\sum_{S\subseteq \Ddigits}(-1)^{\abs{S}}\left(C(S)\,X+O\!\left(\frac{X}{\sqrt{\log X}}\right)\right)
=\left(\sum_{S\subseteq \Ddigits}(-1)^{\abs{S}}C(S)\right)X+O\!\left(\frac{X}{\sqrt{\log X}}\right).
\]
The constant in parentheses is exactly $c_{\mathrm{dead}}$ as defined in Theorem~\ref{thm:main}.
Since there are only $2^{10}=1024$ subsets $S$,
the error term remains $O\!\left(\frac{X}{\sqrt{\log X}}\right)$.
This completes the proof.
\qed

\subsection{Proof of Theorem~\ref{thm:base_b}}
The proof is a direct adaptation of the base-10 argument in Theorem~\ref{thm:main},
with the digit set and the relevant
congruences replaced by their base-$b$ analogues. Here we explain how this is done.

Fix an integer base $b\geq 2$ and write
\[
\mathcal{D}_b:=\{0,1,\dots,b-1\}.
\]
A positive integer $N$ is a base-$b$ dead end if $N$ is square-free and $bN+d$ is not square-free for every
$d\in\mathcal{D}_b$. Using $\mathbf{1}_{m\ \mathrm{square\text{-}free}}=\mu(m)^2,$ we obtain the modified indicator
identity
\[
\mathbf{1}_{N\ \mathrm{is\ a\ base\text{-}b\ dead\ end}}
=\mu(N)^2\prod_{d\in\mathcal{D}_b}\bigl(1-\mu(bN+d)^2\bigr),
\]
and expanding the product yields the inclusion--exclusion formula
\[
D_b(X)=\sum_{S\subseteq\mathcal{D}_b}(-1)^{|S|}\sum_{N\le X}\mu(N)^2\prod_{d\in S}\mu(bN+d)^2.
\]

For each prime $p$ and $S\subseteq\mathcal{D}_b$, define the local obstruction count
\[
\nu_{p,b}(S):=\#\Bigl\{n\bmod p^2:\ p^2\mid n\ \text{or}\ p^2\mid (bn+d)\ \text{for some }d\in S\Bigr\}.
\]
Exactly as in the proof of Theorem~\ref{thm:main},
one analyzes the congruences modulo $p^2$ and uses the Chinese remainder theorem
to conclude that the main term for each fixed $S$ is an Euler product with local factor
$1-\nu_{p,b}(S)/p^2$. Summing over $S$ gives the constant $c_{\mathrm{dead}}(b)$
from Theorem~\ref{thm:base_b}.

If $p\nmid b$, then $b$ is invertible modulo $p^2$ and each condition
$p^2\mid (bn+d)$ is equivalent to the single
congruence
\[
n\equiv -d\,b^{-1}\pmod{p^2}.
\]
Thus, for primes with $p^2>b$ these residue classes are distinct as $d$ varies,
so (up to the overlap when $d=0$)
one has the same phenomenon as in base $10$: a single prime square can obstruct at most one digit at a time.
This yields the base-$b$ analogue of the computation of $\nu_{p,b}(S)$ for all
sufficiently large primes $p\nmid b$.

When $p\mid b$, $b$ is not invertible modulo $p^2$,
so the congruence $p^2\mid (bn+d)$ requires separate treatment
(the analogue of our special handling of $p=2$ and $p=5$ in base $10$).
Likewise, for the finitely many primes with
$p^2\le b$, distinct digits may collide modulo $p^2$,
and these are again handled by a direct residue-class count.
Since there are only finitely many such primes (depending on $b$),
this only modifies finitely many Euler factors and does
not affect the structure of the argument.

Combining the inclusion--exclusion expansion with the prime-by-prime sieve computation,
we obtain the stated asymptotic
\[
D_b(X)=c_{\mathrm{dead}}(b)\,X+O_b\!\left(\frac{X}{\sqrt{\log X}}\right),
\]
with $c_{\mathrm{dead}}(b)$ given by the finite alternating sum of Euler products in Theorem~\ref{thm:base_b}.
\qed

\section{AxiomProver}\label{sec:AxiomProver}
At Axiom Math, we are developing AxiomProver, an AI system for mathematical research via formal proof.
As an early test case in this effort, we present this case study,
treating Theorem~\ref{thm:base_b} as an end-to-end formalization target.
We chose them because we believed that they are within reach of today's Mathlib.

This paper confirms that expectation: the proof is fully formalized in Lean/Mathlib
(see \cite{Lean,Mathlib2020})
and was produced by AxiomProver from a natural-language statement
of the Main Problem (without the answer given).
We now make precise what ``produced'' means, and describes the end-to-end pipeline we used.

\subsection*{Process}
The formal proofs provided in this work were developed and verified using \textbf{Lean 4.26.0}.
Compatibility with earlier or later versions is not guaranteed due to the evolving nature of the Lean 4 compiler and its core libraries.
The relevant files are all posted in the following repository:
\begin{center}
  \url{https://github.com/AxiomMath/dead-ends}
\end{center}
The input files were
\begin{itemize}
  \item \texttt{dead-ends.tex}, a natural-language statement of the Main Problem for all bases $b \ge 2$
    (i.e.\ finding a formula for $c_{\mathrm{dead}}(b)$ and proving its correctness);
  \item a \texttt{task.md} that contains the single line
  \begin{quote}
    \slshape
    Read dead-ends.tex.
  \end{quote}
  \item a configuration file \texttt{.environment} that contains the single line
  \begin{quote}
    \slshape
    lean-4.26.0
  \end{quote}
  which specifies to AxiomProver which version of Lean should be used.
\end{itemize}
Given these three input files,
AxiomProver autonomously provided the following output files:
\begin{itemize}
  \item \texttt{problem.lean}, a Lean 4.26.0 formalization of the problem statement
which includes the autonomously computed formula for $c_{\mathrm{dead}}(b)$; and
  \item \texttt{solution.lean}, a complete Lean 4.26.0 formalization of the proof.
\end{itemize}

After AxiomProver generated a solution, the human authors wrote this paper
(without the use of AI) for human readers.
Indeed, a research paper is a narrative designed to communicate ideas to humans,
whereas a Lean files are designed to satisfy a computer kernel.

\section{Appendix: Computing $\nu_p(S)$ and $c_{\mathrm{dead}}$}\label{Appendix:Numerics}

This appendix records two additional facts:
\begin{itemize}
\item[(A)] one can compute $\nu_p(S)$ explicitly for each prime $p$ and each digit subset $S$;
\item[(B)] using these formulas, one can evaluate $c_{\mathrm{dead}}$ numerically to high precision.
\end{itemize}

\subsection{Explicit computation of $\nu_p(S)$}

Here we explicitly compute the $\nu_p(S)$. For primes $p\geq 7,$ this reduces to whether or not $0$ is in $S$.

\begin{proposition}[Primes $p\ge 7$]\label{prop:nup-large}
Let $p\ge 7$ be prime and let $S\subseteq\Ddigits$. Then $10$ is invertible modulo $p^2$ and
\[
\nu_p(S)=1+\abs{S}-\1_{0\in S}.
\]
\end{proposition}

\begin{proof}
Modulo $p^2$, the congruence $p^2\mid n$ forces $n\equiv 0$ (one residue class).
For each $d\in S$, the congruence $p^2\mid(10n+d)$ is equivalent to
\[
n\equiv -d\cdot 10^{-1}\pmod{p^2},
\]
so it specifies exactly one residue class.
If $d\neq d'$, then $d\not\equiv d'\pmod{p^2}$ because $p^2\ge 49>9$, so these residue classes are distinct.
The class for $d=0$ coincides with $n\equiv 0$, giving the stated union size.
\end{proof}

The primes $2,3,5$ are special because $10$ is not invertible modulo $4$ or $25$,
and because $9\le 10$ introduces a collision modulo $9$.

\begin{proposition}[The prime $p=3$]\label{prop:nup3}
Since $10\equiv 1\pmod 9$,
\[
\nu_3(S)=\left|\{0\}\cup\{-d\bmod 9:\ d\in S\}\right|.
\]
\end{proposition}

\begin{proof}
We have $9\mid(10n+d)$ if and only if $9\mid(n+d)$, i.e.\ $n\equiv -d\pmod 9$.
Taking the union over $d\in S$ and adding $n\equiv 0$
(from $9\mid n$) gives the formula.
\end{proof}

\begin{proposition}[The prime $p=2$]\label{prop:nup2}
The following are true.
\begin{enumerate}
\item if $d$ is odd, then $4\nmid (10n+d)$ for all $n$ (no solutions);
\item if $d\equiv 0\pmod 4$ (i.e.\ $d\in\{0,4,8\}$), then $4\mid(10n+d)$ holds for $n\equiv 0,2\pmod 4$;
\item if $d\equiv 2\pmod 4$ (i.e.\ $d\in\{2,6\}$), then $4\mid(10n+d)$ holds for $n\equiv 1,3\pmod 4$.
\end{enumerate}
Therefore, $\nu_2(S)$ is the size of the union of these residue sets together with $\{0\}$ (from $4\mid n$).
\end{proposition}

\begin{proof}
Solve $2n+d\equiv 0\pmod 4$ case-by-case. If $d$ is odd, the left-hand side is odd so has no solutions.
If $d$ is even, divide by $2$ to obtain a condition modulo $2$, which lifts to two classes modulo $4$.
\end{proof}

\begin{proposition}[The prime $p=5$]\label{prop:nup5}
The following are true:
\begin{enumerate}
\item $25\mid (10n)$ if and only if $5\mid n$, i.e.\ $n\equiv 0\pmod 5$ (five classes modulo $25$);
\item $25\mid (10n+5)$ if and only if $n\equiv 2\pmod 5$ (five classes modulo $25$).
\end{enumerate}
Therefore, $\nu_5(S)$ is the size of the union of these residue sets together with $\{0\}$ (from $25\mid n$).
\end{proposition}

\begin{proof}
If $5\nmid d$, then $10n+d$ is not divisible by $5$, hence cannot be divisible by $25$.
If $d=0$, then $25\mid 10n$ is equivalent to $5\mid n$.
If $d=5$, then $25\mid (10n+5)=5(2n+1)$ is equivalent to $5\mid (2n+1)$, i.e.\ $n\equiv 2\pmod 5$.
Each congruence class modulo $5$ lifts to exactly five classes modulo $25$.
\end{proof}

\subsection{Computing $c_{\mathrm{dead}}$ numerically}

Using the explicit formulas above, one can compute $c_{\mathrm{dead}}$ to high precision.
A convenient simplification is the observation that for every prime $p\ge 7$
the local factor depends on $S$ only through $\abs{S}$ and whether $0\in S$.
This reduces the $1024$ Euler products to a linear combination of only $22$ products.

The decimal shown in Theorem \ref{thm:main} was obtained by evaluating these products using
the rapidly convergent identity, valid for $p\ge 7$ and $1\le a\le 9$,
\[
\sum_{p\ge 7}\log\!\left(1-\frac{a}{p^2}\right)
=-\sum_{k\ge 1}\frac{a^k}{k}\sum_{p\ge 7}\frac{1}{p^{2k}},
\]
and truncating the $k$-sum at a point where the tail is far below the displayed precision.
(The ratio $a/p^2\le 9/49$ ensures very fast geometric decay.)

\medskip


\begin{thebibliography}{9}

\bibitem{Apostol}
T.~Apostol,  \emph{Introduction to Analytic Number Theory}, New York, Springer, 1976.

\bibitem{Kominers} S.~D.~Kominers, \emph{Uniform bounds for digit-appending Fibonacci walks},
(https://arxiv.org/abs/2512.06446)

\bibitem{Lean}
L.~de~Moura, S.~Kong, J.~Avigad, F.~van~Doorn, and J.~von~Raumer,
The {L}ean theorem prover (system description),
in \emph{Automated Deduction -- CADE-25},
Lecture Notes in Computer Science~9195, Springer, 2015, 378--388.

\bibitem{Mathlib2020}
The mathlib Community,
The {L}ean mathematical library,
in \emph{Proceedings of the 9th ACM SIGPLAN International Conference on
Certified Programs and Proofs (CPP 2020)},
ACM, 2020.

\bibitem{MillerEtAl}
S.~Miller, Y.~Peng, I.~Popescu, K.~\c{S}iktar, and S.~Wattanawanichkul,
\emph{Walking to infinity along number theory sequences}, Integers \textbf{24} (2024).

\bibitem{MIRSKY}
L.~Mirsky, \emph{Note on an asympotic formula connected with $r$-free integers}, The Quarterly Journal of Mathematics, Oxford, \textbf{18} (1947), 178-182.

\end{thebibliography}
\end{document}